\documentclass[letterpaper, 10 pt, conference]{ieeeconf}
\IEEEoverridecommandlockouts

\usepackage{cite}
\usepackage{float}
\usepackage{amsmath,amssymb,amsfonts}
\usepackage{algorithmic}
\usepackage{mathrsfs}
\usepackage{graphicx}
\usepackage{textcomp}
\usepackage{xcolor}
\usepackage[amsmath,thmmarks]{ntheorem}
\usepackage{hyperref}
\usepackage{enumerate}
\usepackage{nccmath}
\usepackage[normalem]{ulem}
\hypersetup{pdfborder={0 0 0},}
\def\BibTeX{{\rm B\kern-.05em{\sc i\kern-.025em b}\kern-.08em
    T\kern-.1667em\lower.7ex\hbox{E}\kern-.125emX}}
\usepackage{algorithm}

 \theoremheaderfont{\normalfont\bfseries} 
\theorembodyfont{\normalfont\itshape}    
\theoremseparator{.}                     

\newtheorem{remark}{Remark}
\newtheorem{theorem}{Theorem}
\newtheorem{lemma}{Lemma}

\newtheorem{assumption}{Assumption}

\newcommand{\ADDBoYu}[1]{\textcolor{black}{{#1}}}
\newcommand{\ADDBoyu}[1]{\textcolor{black}{{#1}}}

\begin{document}

\title{\LARGE \bf Decentralized Linearized Consensus ADMM with Efficient Quantized Communication
}

\author{Boyu Han, Xu Du,  Karl~H.~Johansson, and Apostolos I. Rikos$^*$
\thanks{$^*$Corresponding author.}
	\thanks{Boyu Han, Xu Du and Apostolos I. Rikos are with the Artificial Intelligence Thrust of the Information Hub, The Hong Kong University of Science and Technology (Guangzhou), Guangzhou, China. 
    Apostolos I. Rikos is also affiliated with the Department of Computer Science and Engineering, The Hong Kong University of Science and Technology, Clear Water Bay, Hong Kong, China. E-mails: {\tt~\{boyuhan, michaelxudu, apostolosr\}@hkust-gz.edu.cn}. 
            }
            \thanks{Karl H.~Johansson is with the Division of Decision and Control Systems, KTH Royal Institute of Technology, SE-100 44 Stockholm, Sweden. 
    He is also affiliated with Digital Futures, SE-100 44 Stockholm, Sweden. 
    E-mail:{\tt~kallej@kth.se}.
            }
            \thanks{The work of B.H., X.D. and A.I.R. was supported by the Guangzhou-HKUST(GZ) Joint Funding Scheme (Grant No. 2025A03J3960). The work of A.I.R. was also supported by the Guangdong Provincial Project (Grant No. 2024QN11G109).} 
}

\maketitle

\begin{abstract}

Distributed optimization offers significant advantages over centralized methods in terms of scalability and robustness when solving large-scale problems.
In this paper, we propose a novel decentralized optimization algorithm that integrates Inexact Consensus ADMM (IC-ADMM) with a finite-time decentralized quantized communication algorithm.
The proposed method enjoys three main benefits: (i) it operates on directed communication graphs, (ii) it requires only quantized local information instead of exact values, and (iii) it does not rely on solving the local subproblems exactly.
Under the assumption that each node’s local objective is $\mu$-strongly convex and $L$-smooth, the algorithm is guaranteed to achieve global linear convergence to a neighborhood of the optimal solution.
Extensive numerical experiments demonstrate its advantages over existing methods in terms of convergence speed and communication efficiency.

\end{abstract}

\section{Introduction}
Distributed optimization has found broad applications in machine learning \cite{zhou2026preconditioned}, wireless communication \cite{zhang2025enabling}, model predictive control \cite{camacho2026introduction}, and power networks \cite{lanza2025distributed}.
As these applications grow in scale and complexity, centralized computation often encounters limitations in hardware resources, memory capacity, and computational efficiency \cite{2024_doostmohammadian_rikos_Johansson_survey}.
Moreover, centralized architectures are vulnerable to single points of failure \cite{dist_structure3, dist_structure4, dist_structure5}.
By contrast, distributed optimization enables multiple agents to collaboratively solve large-scale problems through local computation and information exchange over communication networks, thereby enhancing scalability, efficiency and robustness.

Distributed optimization algorithms are commonly classified into two main categories \cite{boyd2007notes}:
(i) primal decomposition and (ii) dual decomposition.
In both frameworks, each node updates its local decision variables while exchanging information with neighboring nodes or a central coordinator.
In primal decomposition approaches, the exchanged information typically consists of local primal variables, (sub)gradients, or second-order information such as Hessian approximations.
In contrast, dual decomposition methods additionally introduce dual variables associated with coupling constraints, which are updated and exchanged across the network.
In this work, we focus on dual decomposition-based methods, particularly the Alternating Direction Method of Multipliers (ADMM) \cite{boyd2011distributed}, which has attracted significant attention due to its effectiveness and wide applicability in large-scale distributed optimization problems.

\noindent
\textbf{Existing Literature.}   ADMM can be applied to solve several classes of distributed optimization problems, including
(i) distributed resource allocation and (ii) distributed consensus optimization \cite{boyd2011distributed}.
This paper focuses on the latter class.
The ADMM iteration typically consists of three main steps:
(i) local primal variable updates,
(ii) information exchange among nodes, and
(iii) dual variable updates.
In the first step, the local subproblems can be solved either exactly or inexactly to reduce the computational burden per iteration. 
For distributed resource allocation problems, linearized or inexact  ADMM schemes were  proposed in \cite{ng2011inexact}, \cite{he2020optimally}, \cite{chen2017efficient} and \cite{jang2026aliaadaptivelinearizedadmm}.
The global linear convergence of linearized ADMM for resource allocation problems was later established in \cite{cao2018distributed}.
For distributed consensus problems, the works \cite{ling2015dlm}, \cite{mokhtari2016dqm}, and \cite{li2019communication} established global linear convergence of Inexact Consensus ADMM (IC-ADMM) under the assumptions of strong convexity and Lipschitz continuous gradients.
Notably, these approaches employ a fully decentralized architecture that eliminates the central coordinator, offering greater structural robustness in distributed networks.
From a different perspective, \cite{11373371} analyzed the convergence behavior of IC-ADMM in terms of regret bounds.
More recently, \cite{yue2025differentially} extended the convergence analysis of decentralized IC-ADMM to the smooth nonconvex setting.
However, these methods exhibit two limitations.
First, they require information exchange over undirected communication graphs, which restricts their applicability in networks with asymmetric communication links.
Second, the exchanged information typically contains the exact value of local variables, resulting in significant communication overhead.
To address the first limitation, \cite{jiang2021distributed} proposed the Decentralized-ADMM-FTERC (D-ADMM-FTERC) algorithm, which extends Consensus ADMM (C-ADMM) to convex optimization over directed graphs.
Nevertheless, its convergence guarantees hold only when the local subproblems are solved exactly.
Furthermore, the algorithm still requires the exchange of exact local variables, and thus does not reduce the communication cost. To address the second limitation, \cite{xu2023qc} introduced decentralized quantization technique for information exchange among nodes.
However, their method still relies on the assumption of undirected communication graphs.
To the best of our knowledge, within the C-ADMM framework, only \cite{rikos2023asynchronous} and \cite{du2025cdc} address all limitations simultaneously, i.e., communication over directed graphs, quantized information exchange and decentralized fashion.
However, their convergence guarantees rely on the exact solution of the local subproblems.
This indicates that IC-ADMM algorithms with quantized communication over directed graphs remain largely unexplored.

\noindent
\textbf{Contribution.} Motivated by the aforementioned literature, we propose a novel \emph{decentralized} algorithm that integrates IC-ADMM over directed communication graphs while relying solely on quantized local information for inter-node communication.
To the best of our knowledge, this is the first work that enables IC-ADMM to operate under the following requirements simultaneously:
(i) directed communication graphs,
(ii) quantized information exchange, and
(iii) fully decentralized computation.
The main contributions of this paper are summarized as follows.

\noindent
\textbf{A.} Inspired by \cite{rikos2025distributed} and \cite{du2025cdc}, we propose a two-layer decentralized algorithm, termed Decentralized Quantized Linearized Consensus ADMM (DQLCA), which consists of an optimization layer and a communication layer.
The communication layer, as in \cite{rikos2022non}, exchanges information over directed graphs using only quantized local variables.
The outer optimization layer updates the local decision variables inexactly following the IC-ADMM scheme and simultaneously updates the corresponding local dual variables. 
We show that our proposed algorithm achieves global linear convergence to a neighborhood of the optimal solution under the assumptions that each node’s local cost function is $\mu$-strongly convex and $L$-smooth.
\\
\textbf{B.} In extensive numerical experiments we demonstrate its advantages over existing methods in terms of computation time and communication efficiency.

\section{Notation and Background}

In this paper, $\mathbb{R}$, $\mathbb{Q}$, and $\mathbb{Z}$ denote the sets of real, rational, and integer numbers, respectively. Matrices are represented by capital letters (e.g., $A\in \mathbb{R}^{n \times n}$), and vectors by lowercase letters (e.g., $a \in \mathbb{R}^n$). The symbols $A^\top$ and $a^\top$ denote the transpose of a matrix $A$ and a vector $a$, respectively. For a matrix $A$, the entry in the $i$-th row and $j$-th column is denoted by $a_{ij}$. 
\ADDBoYu{For any real vector $a \in \mathbb{R}^n$, let $\lfloor a \rfloor \in \mathbb{Z}^n$ and $\lceil a \rceil \in \mathbb{Z}^n$ denote the component-wise floor and ceiling of $a$, respectively; that is, $\lfloor a \rfloor \leq a \leq \lceil a \rceil$ holds component-wise.}
The symbol $\|a\|$ denotes the Euclidean norm of a vector $a$, while $\|a\|_\infty$ denotes the infinity norm.  Furthermore, $|\mathcal S|$ denotes the cardinality of a countable set $\mathcal S$.

\subsection{Graph-Theoretic Notions} 
In this paper, $\mathcal{V} = \{1,2,\dots,N\}$ denotes the set of nodes, and $\mathcal{E} \subseteq \mathcal{V} \times \mathcal{V} \cup \{(i,i) \mid i \in \mathcal{V}\}$ represents the set of edges, where $(i,j) \in \mathcal{E}$ indicates that node $i$ can receive information from node $j$. The communication network is modeled as a directed graph (digraph) $\mathcal{G} = (\mathcal{V}, \mathcal{E})$, with $N = |\mathcal{V}| \geq 2$ being the number of nodes. For each node $i \in \mathcal{V}$, the set of out-neighbors is defined as $\mathcal{N}_i^+ = \{ l \in \mathcal{V} \mid (l, i) \in \mathcal{E} \}$, and the set of in-neighbors as $\mathcal{N}_i^- = \{ j \in \mathcal{V} \mid (i,j) \in \mathcal{E} \}$. The out-degree of node $i$ is denoted by $\mathcal{D}_i^+ = |\mathcal{N}_i^+|$, and its in-degree by $\mathcal{D}_i^- = |\mathcal{N}_i^-|$. The diameter $D$ of the digraph $\mathcal{G}$ is the length of the longest shortest directed path between any two nodes $i, j \in \mathcal{V}$. The digraph $\mathcal{G}$ is said to be strongly connected if there exists a directed path from every node $i$ to every node $j \in \mathcal{V}$.


\subsection{Quantizers}\label{quantizers_subsec}

In communication systems, quantization techniques serve to improve communication efficiency and reduce the required bandwidth. 
\ADDBoYu{By decreasing the number of bits necessary to represent information, they bridge the gap between ideal theoretical optimization and the physical realities of digital networks.}
\ADDBoYu{This capability is critical for preventing channel congestion and ensuring that decentralized algorithms remain practically implementable even under severe capacity constraints \cite{rikos2025distributed}.} 
In this paper, we employ an asymmetric mid-rise quantizer with infinite range, defined as
\begin{equation}\label{eq: quantizer}
q_{\Delta}^a(b) =  \left\lfloor \frac{b}{\Delta} \right\rfloor \in \mathbb Q^n,
\end{equation}
for a vector $b\in \mathbb R^n$. Here, the superscript $a$ denotes the asymmetric type, and $\Delta \in \mathbb{R}_+$ denotes the quantization level. 
The output is obtained by scaling the floor operation by $\Delta$.
 


\section{Problem Formulation}



\subsection{Distributed Consensus Optimization Problem}


Let us consider a distributed network modeled as a digraph $\mathcal{G} = (\mathcal{V}, \mathcal{E})$ comprising $N=|\mathcal{V}|$ nodes. Each node $i \in \mathcal{V}$ is endowed with a local cost function $f_i: \mathbb{R}^n\to\mathbb{R}$, which is exclusively known to node $i$.
A fundamental problem in this distributed setting is consensus problem, where nodes in the network cooperate to solve a global optimization task while requiring all nodes to reach a unanimous agreement. The standard formulation of this problem, which inherently relies on a central coordinator, is given by:

\begin{equation}\label{eq: consensus problem}
    \begin{aligned}
        \min_{x_i,z\in \mathbb R^n} & \quad \sum_{i=1}^{N} f_i(x_i)\\
        \text{s.t.}\;\; & \;\quad x_i =z, \quad\forall i = 1, ..., N,
    \end{aligned}
\end{equation}
where $x_i\in \mathbb R^n$ is the local variable for node $i$ and $z\in \mathbb R^n$ is the global consensus variable. Notably, the constraint $x_i=z$ still implicitly requires a central coordinator, motivating the need for a fully decentralized network model.
\subsection{Decentralized Quantized Consensus Optimization Problem over Network}
In realistic distributed systems, such as sensor networks or multi-agent robotics, there is no central coordinator. Moreover, these physical nodes communicate solely over bandwidth-limited channels, meaning local variables cannot be transmitted with infinite precision. Consequently, while the local objective functions $f_i$ are optimized over a continuous domain, any information exchanged between a node $i$ and its neighbor $j$ is restricted. To address this, we quantize the transmitted messages using a quantization level $\Delta$. Accordingly, the decentralized quantized consensus problem is formulated as:
\begin{equation}\label{eq: quantized problem}
    \begin{aligned}
        \quad\qquad\min_{x_i,z\in\mathbb R^n} & \quad \sum_{i=1}^{N} f_i(x_i)\\
        \text{s.t.}\;\; & \;\quad x_i =z, \quad\forall i = 1, ..., N \\
        &\;\quad \text{nodes communicate with quantized values,}
    \end{aligned}
\end{equation}
where the quantized values are based on the quantization level $\Delta$. Also, let $z^*$ denotes the optimal solution to the problem \eqref{eq: quantized problem}. While the presence of the global consensus variable $z$ in \eqref{eq: quantized problem} ostensibly suggests a centralized architecture, our objective is to solve this problem in a \ADDBoYu{fully} decentralized manner. Rather than relying on a central fusion center to enforce $x_i = z$, our proposed algorithmic framework in the next section will achieve this global agreement through purely local, peer-to-peer exchanges over the digraph $\mathcal{G}$. Consequently, the central variable $z$ serves only as an analytical anchor for our formulation, while the actual computation and quantized communication are executed entirely at the node level.  

\subsection*{Assumptions}
Before proceeding to the algorithmic development, we establish the following fundamental assumptions regarding the network topology and the local cost functions.
\begin{assumption}\label{ass:3}
  
    The communication network is modeled as a \textit{strongly connected} digraph $\mathcal{G} = (\mathcal{V}, \mathcal{E})$. 
    Also, every node $i$ knows the diameter of the network $D$, and a common quantization level $\Delta$. 
\end{assumption}

\begin{assumption}\label{ass:1} 
The local cost function $f_i$ of each node $i \in \mathcal{V}$ is closed, proper, $L$-smooth and $\mu$-strongly convex. 
Specifically, for each local cost function $f_i$, for every $x_\alpha,x_\beta \in \mathbb R^n$, there exists a \emph{strong-convexity} constant $\mu_i>0$ and a \emph{Lipschitz-continuity} constant $L_i>0$ such that 
\begin{equation}\label{eq: mu}
\begin{aligned}
    f_i(x_\alpha)+&\nabla f_i(x_\alpha)^\top (x_\beta-x_\alpha) +\frac{\mu_i}{2}\|x_\beta-x_\alpha\|^2\leq f_i(x_\beta),
\end{aligned}
\end{equation}
and
\begin{equation}\label{eq: lip}
\begin{aligned}  
     \left\|\nabla f_i(x_\alpha)-\nabla f_i(x_\beta) \right\| \leq L_i \left\|x_\alpha-x_\beta\right\|,
\end{aligned}
\end{equation}
hold.
\end{assumption}




The assumptions introduced above lay the groundwork for the subsequent global convergence analysis.
First, the strongly connected digraph defined in Assumption \ref{ass:3} facilitates the achievement of finite-time quantized consensus across the network. \ADDBoyu{Moreover, the diameter $D$ of the digraph can be computed by each node in a distributed and finite-time fashion via the protocol proposed in \cite{oliva2016distributed}.} Second, the $\mu$-strong convexity defined in \eqref{eq: mu} and $L$-smooth property in \eqref{eq: lip} guarantee the existence of a unique global optimal solution for problem \eqref{eq: quantized problem}. Third, the $L$-smoothness provided in \eqref{eq: lip} will be used in Lemma~\ref{lemma:2} to bound the error term in \eqref{error:3}. It is worth noting that, Assumption \ref{ass:1} represents standard, widely adopted conditions in the design of decentralized optimization algorithms \cite{ling2015dlm}. Together, these assumptions formulate the rigorous mathematical foundation required to establish the global convergence analysis of Algorithm \ref{alg: DQLCA}, see Theorem~\ref{thm: convergence}.


\section{Overview of Linearized Consensus ADMM}
\subsection{Exact Consensus ADMM}

The exact C-ADMM offers a robust theoretical framework in \cite[Chapter~7]{boyd2011distributed}. 
We first construct the Augmented Lagrangian for \eqref{eq: consensus problem}, which is given by:
\begin{equation}\label{eq: augmented lagrangian}
    \mathcal{L}_{\rho}(x, z, \lambda) \overset{\text{def}}{=} \sum_{i=1}^{N} \left( f_{i}(x_{i}) + \lambda_{i}^{\top} (x_{i} - z) + \frac{\rho}{2} \|x_{i} - z\|^{2} \right),
\end{equation}
where $\lambda_i \in \mathbb R^{n}$ is the Lagrangian multiplier associated with node $i$, the concatenated dual variable is $\lambda = [\lambda_1^{\top}, \lambda_2^{\top}, ..., \lambda_N^{\top}]^{\top},$ the concatenated primal variable is $x=[x_1^{\top}, x_2^{\top}, ..., x_N^{\top}]^{\top}$ and $\rho$ is the positive penalty parameter. 
The exact C-ADMM (see \cite{boyd2011distributed}) iteratively updates the primal variables, the consensus variable, and the dual variables through the following three-step process at each iteration $k$:
\begin{subequations}\label{eq: steps of standard ADMM}
\begin{align}
    x_{i}^{[k+1]} &= \underset{x_{i}}{\arg\min}\; \mathcal{L}_{\rho}\left(x_{i}, z^{[k]}, \lambda_{i}^{[k]}\right) \label{eq:step1} \\
    z^{[k+1]}     &= \underset{z}{\arg\min}\; \mathcal{L}_{\rho}\left(\mathbf{x}^{[k+1]}, z, \lambda^{[k]}\right) \label{eq:step2} \\
    \lambda_{i}^{[k+1]} &= \lambda_{i}^{[k]} + \rho \left( x_{i}^{[k+1]} - z^{[k+1]} \right) \label{eq:step3}
\end{align}
\end{subequations}

\subsection{Inexact Consensus ADMM}

While this standard framework guarantees convergence, solving the exact primal update in \eqref{eq:step1} can still impose a considerable computational burden at each node. To alleviate this, linearized C-ADMM is proposed in \cite{hong2016convergence,zhou2023federated}. Instead of minimizing $f_i$ directly, each local cost function is approximated by its first-order Taylor expansion:
\begin{equation}\label{eq: first-order taylor expansion}
f_{i}(\zeta) \overset{\text{def}}{=} f_{i}(\xi) + \nabla f_{i}(\xi)^{\top}(\zeta - \xi)+\mathcal O\left(\left\|\zeta-\xi\right\|^2\right),\; \forall i,
\end{equation}
where $\zeta,\xi \in \mathbb R^n$. 
Substituting the linear approximation \eqref{eq: first-order taylor expansion} of $f_i(x_i)$ into the first step \eqref{eq:step1} of exact  C-ADMM yields the computationally efficient, closed-form $x$-update process of linearized C-ADMM, with \eqref{eq:step2} derived to its closed form \cite[Chapter~7]{boyd2011distributed} and \eqref{eq:step3} unchanged:

\begin{subequations}\label{eq: steps of linearized ADMM closed-form}
\begin{align}
    x_{i}^{[k+1]} &= z^{[k]} - \frac{1}{\rho} \left( \nabla f_i\left(z^{[k]}\right) + \lambda_i^{[k]} \right), \label{eq:step1-closed} \\
    z^{[k+1]} &= \frac{1}{N} \sum_{i=1}^{N} \left( x_i^{[k+1]} + \frac{1}{\rho}\lambda_i^{[k]} \right), \label{eq:step2-closed} \\
    \lambda_i^{[k+1]} &= \lambda_i^{[k]} + \rho \left( x_i^{[k+1]} - z^{[k+1]} \right). \label{eq:step3-closed}
\end{align}
\end{subequations}
Compared to \eqref{eq: steps of standard ADMM}, only the local primal update \eqref{eq:step1-closed} is fundamentally modified through linearization. \ADDBoYu{In exact C-ADMM, updating the primal variable often requires solving a complex optimization subproblem at each node, which typically necessitates running computationally expensive solvers. By linearizing the local cost function, we transform this optimization subproblem into a simple, closed-form equation. Consequently, each node simply evaluates a direct algebraic formula at each iteration, completely eliminating the need for complex solvers and significantly decreasing the computational burden.}

It is noticeable that \eqref{eq: steps of linearized ADMM closed-form} still requires the global consensus variable $z$ in \eqref{eq:step2-closed}, which can not be solved in a decentralized manner. Inspired by this limitation, we will develop a fully decentralized algorithm in the subsequent section. 
Our algorithm is capable of (i) enabling each node to update its local variable only communicating with its immediate neighbors (no more need for global variable or central coordinator), and (ii) achieving efficient communication within the directed network. 


\section{Decentralized Linearized Consensus ADMM with efficient Communication}

In this section, we propose a novel algorithm, termed DQLCA, designed to solve \ADDBoYu{decentralized} optimization problems over strongly connected digraphs \ADDBoyu{consisting} of bandwidth-limited communication links. 

\subsection{DQLCA Algorithm Development}\label{subsec:Algorithm Development} 
To formulate a \ADDBoYu{fully} decentralized algorithm, it is necessary to eliminate the reliance on the centralized consensus variable $z$ presented in \eqref{eq: steps of linearized ADMM closed-form}. We achieve this by assigning each node $i \in \mathcal{V}$ a local tracking variable, denoted as $\hat z_i$, which serves as the node's local estimate of global state. 
A local primal variable $\hat x_i$ and a local dual variable $\hat\lambda_i$ are introduced for each node $i$, aiming to distinguish the quantized inexact iterates in Algorithm \ref{alg: DQLCA} from the exact variables in \eqref{eq: steps of linearized ADMM closed-form}.
Following this, we present our proposed algorithm (detailed below as Algorithm~\ref{alg: DQLCA}).

\begin{algorithm}[h] 
	\caption{DQLCA: Decentralized Quantized Linearized Consensus ADMM}
   	\textbf{Input:} A strongly connected digraph $\mathcal{G}$ with $N = |\mathcal{V}|$ nodes and $m = |\mathcal{E}|$ edges. For every node $i \in \mathcal{V}$: digraph diameter $D$, local cost function $f_i$, quantization level $\Delta \in \mathbb{Q}$ and Augmented Lagrangian parameter $\rho$. Assumptions \ref{ass:3}, \ref{ass:1} hold.
    
    \textbf{Initialization:} Randomly chosen $\hat x_i^{[0]}\in \mathbb{R}^n$, $\hat \lambda_i^{[0]} \in \mathbb{R}^n$ that satisfies $\sum_{i=1}^N \hat \lambda_i^{[0]}=0 $ and $\hat z_i^{[0]} \in \mathbb{R}^n$, for each node $i \in \mathcal{V}$. 
    \\
    \textbf{Iteration:} 
	For $k = 0,1,2,\ldots$, each node $i \in \mathcal{V}$ does:
	\begin{enumerate}    
	\item  Optimize $\hat x_i$ as 
    \begin{equation}\label{eq: new local primal}
    \begin{split}
        \hat x_{i}^{[k+1]} &= \hat z_i^{[k]} - \frac{1}{\rho} \left( \nabla f_i\left(\hat z_i^{[k]}\right) + \hat\lambda_i^{[k]} \right).
    \end{split}
    \end{equation}
    \item Updates the local tracking variable as
    \begin{equation}\label{eq: global estimation}
                \hat z_i^{[k+1]} = \text{Algorithm \ref{alg:QuAS}}\left(\left(\hat x_i^{[k+1]}+ \frac{\hat\lambda_i^{[k]}}{\rho}\right),\;D,\;\Delta\right).
    \end{equation}
    \item Update the dual variable as: \begin{equation}\label{eq: dual update 1}
        \hat\lambda_i^{[k+1]} = \hat\lambda_i^{[k]} + \rho \left(\hat x_i^{[k+1]} -  \hat z_i^{[k+1]}\right).
    \end{equation}
	\end{enumerate}
    \textbf{Output.} Each node $i \in \mathcal{V}$ calculates $z^*$ which solves problem \eqref{eq: quantized problem}. 
	\label{alg: DQLCA}
\end{algorithm}

Our algorithm contains three primary steps.   
In the first step, each node locally updates its primal variable $\hat x_i$ via a computationally efficient closed-form expression \eqref{eq: new local primal}. 
This expression is derived by minimizing the Augmented Lagrangian, wherein the local cost function $f_i$ is replaced by its first-order Taylor expansion evaluated at the current local tracking variable $\hat z_i^{[k]}$. Subsequently, in the second step, all nodes collaboratively update their local tracking variables $\hat z_i$ through the decentralized quantized consensus protocol defined in Algorithm \ref{alg:QuAS}, (see \eqref{eq: global estimation}). 
Finally, in the third step, the dual variable of each node is updated using the most recent local primal and tracking variables. 
The algorithm continuously alternates through these three steps until convergence is achieved, ultimately yielding the optimal solution.


Algorithm~\ref{alg:QuAS} follows a structural design analogous to \cite[Algorithm~$1$]{rikos2022non}, consisting of three core components: quantization, averaging and a stopping criterion. In the initialization stage, each node $i$ first quantizes its local information $\scriptstyle\left(x_{i}^{[k+1]}+\lambda_{i}^{[k]}/\rho\right)$ into the quantized value $\chi_i$, which is then split into $\xi_i$ pieces, where a subset of the pieces may have a value exactly one greater than the rest, and the node retains the smallest piece. Subsequently, the node transmits the remaining $(\xi_i-1)$ pieces to randomly selected out-neighbors or itself, while currently receiving pieces $c_j$ from each in-neighbor $j$ to update $\chi_i$ and $\xi_i$ in accordance with equation \eqref{eq: local information avg}. The algorithm additionally performs max-consensus and min-consensus operation every $D$ time steps, producing the respective results $M_i$ and $m_i$. Once the infinity norm of the difference between $M_i$ and $m_i$ is less than or equal to one, each node scales its solution by the quantization level to compute $\hat z_i^{[k+1]}$; Algorithm 2 then terminates, and all nodes transition to step 3 of Algorithm 1. Remarkably, Algorithm 2 ensures finite-time convergence satisfying $||M_i-m_i||_\infty \leq1$, with its convergence time dependent on the network diameter $D$, as validated by \cite{rikos2022non}.

\textbf{Comparison with Previous Works.} The primary contribution of this work is the development of a communication-efficient, fully decentralized linearized algorithm to solve \eqref{eq: quantized problem}. 
Unlike the IC-ADMM proposed in \cite{zhou2023federated}, which relies on a central server and requires the exchange of unquantized, real-valued messages, our DQLCA algorithm achieves full decentralization and efficient communication. 
We eliminate the central coordinator and enhance communication efficiency by introducing the local tracking variable $\hat z_i$ and enforcing quantized communication via Algorithm~\ref{alg:QuAS}. 
Furthermore, compared to the exact quantized C-ADMM framework in \cite{rikos2023asynchronous}, DQLCA incorporates a linearized primal update to significantly reduce the computational burden at each node. 
Finally, in contrast to these prior works, our proposed method guarantees global linear convergence. 
The theoretical details and proofs are provided in the subsequent section.

\begin{algorithm}[h]
	\caption{FTQAC: Finite Time Quantized Average Consensus}
	\textbf{Input:} $\hat x_{i}^{[k+1]}+\hat\lambda_{i}^{[k]}/\rho, D, \Delta$. \\
    \textbf{Initialization.} Each node $i \in \mathcal{V}:$
    \begin{enumerate}
    \item Assigns probability 
    \begin{equation} p_{li}=\left\{
   \begin{aligned}
       &\frac{1}{1+\mathcal D_i^+},\quad& \text{if}\; l\in \mathcal N_i^+\cup \{i\},\\
       &0,\quad & \text{if}\; l\notin \mathcal N_i^+\cup \{i\},
   \end{aligned}
      \right.   
    \end{equation} 
    to each out-neighbor of node $i$.
    \item Sets $\xi_i = 2$, $\chi_i= 2 q^a_{\Delta}(\hat x_{i}^{[k+1]}+\hat\lambda_{i}^{[k]}/\rho)$ (see \eqref{eq: quantizer}).
	\end{enumerate}
	\textbf{Iteration.} For time steps $t=1,2,\cdots$ each node $i \in \mathcal V$ does: 
	\begin{enumerate}
	\item \textbf{If} $t\;\text{mod}(D) = 1$, sets $M_i = \left\lceil \frac{\chi_i}{\xi_i} \right\rceil$ and $m_i=\left\lfloor \frac{\chi_i}{\xi_i} \right\rfloor$.\vspace{2mm}
\item 
Broadcasts $M_i, m_i$ to each out-neighbor $l \in \mathcal N_i^+$ and receives $M_j, m_j$ from each in-neighbor $j \in \mathcal N_i^-$. 
Then, sets $M_i = \text{max}_{j \in \mathcal N_i^- \cup \{i\}}\; M_j$, $m_i = \text{min}_{j \in \mathcal N_i^-\cup \{i\}}\; m_j$. \vspace{2mm}
\item Sets $\tau_i = \xi_i$. \vspace{2mm}
\item \textbf{While} $\tau_i>1$ \textbf{do}
\begin{enumerate}
\item $c_i=\left\lfloor \frac{\chi_i}{\xi_i} \right\rfloor$. 
\item Sets $\chi_i=\chi_i - c_i$, $\xi_i=\xi_i-1$, $\tau_i=\ADDBoyu{\tau_i}-1$. 
\item Transmits $c_i$ to randomly chosen out-neighbor $l \in \mathcal N_i^+\cup \{i\}$ with probability $p_{li}$. 
\item  Receives $c_i$ from $j\in \mathcal N_i^-$ and updates
    \begin{subequations}\label{eq: local information avg}
    \begin{align}
    \chi_i^{[t+1]} &= \chi_i^{[t]} + \sum_{j\in \mathcal N_i^- }w_{ij}^{[t]} c_j^{[t]}, \\
    \xi_i^{[t+1]} &= \xi_i^{[t]} + \sum_{j\in \mathcal N_i^-} w_{ij}^{[t]}.
        \end{align}
    \end{subequations}    
    Here $w_{ij}^{[t]} = 1$ if node $i$ receives $c_j^{[t]}$ from node $j$ at step $t$. 
    Otherwise $w_{ij}^{[t]}=0$ and node $i$ does not receive information from node $j$. 
  \end{enumerate}
    \item \textbf{if} $t\; \text{mod}\; (D)=0$ and $\left\|M_i-m_i\right\|_{\infty}\leq 1$, set $\hat z_i^{[k+1]} = m_i \Delta$, and stop the operation of the algorithm. 
    \end{enumerate}
\textbf{Output.} $\hat z_i^{[k+1]}$.
	\label{alg:QuAS}
\end{algorithm}

\subsection{Convergence Analysis}\label{subsec:Convergence Analysis}

We now establish the theoretical guarantees for the proposed Algorithm~\ref{alg: DQLCA}. 
Our analysis relies on the strict bounds provided by the strong convexity and Lipschitz smoothness of each local cost function, as formalized in Assumptions~\ref{ass:1}. 
Before presenting the main convergence result, we first introduce two foundational lemmas to support the rigorous analysis.

\begin{lemma}\label{lemma:1}
Based on the update rules in Algorithm \ref{alg: DQLCA}, the relationship between the local dual variable $\hat \lambda_i$, the local tracking variable $\hat z_i$ and the gradient of the local cost function evaluated at $\hat z_i^{[k]}$ is given by: 
\begin{equation}\label{eq: sub1}
    \nabla f_i(\hat z_i^{[k]}) = \rho(\hat z_i^{[k]} - \hat z_i^{[k+1]}) - \hat \lambda_i^{[k+1]}.
\end{equation}
\end{lemma}
\textit{Proof.} See Appendix \ref{app 1}.  \hfill$\blacksquare$

\begin{lemma} \label{lemma:2}
The update of the local tracking variable $\hat z_i$ of each node $i \in \mathcal{V}$ is given by Algorithm~\ref{alg: DQLCA} in \eqref{eq: global estimation}. 
According to the constraints of \eqref{eq: quantized problem}, the following equations are satisfied:

\begin{subequations}\label{eq: error}
\begin{align}
    &\hat{z}_i^{[k+1]} = \frac{1}{N} \sum_{i=1}^N \Delta 
      \left\lfloor \frac{\hat x_i^{[k+1]}+\hat{\lambda}_i^{[k]}/\rho}{\Delta} \right\rfloor + \kappa_i, 
      \quad \left\| \kappa_i \right\|_{\infty} \leq \Delta,\label{error:1} \\
    &\left\|z^{[k+1]}-\hat{z}_i^{[k+1]}\right\|_\infty \leq 2\Delta,\label{error:2} \\
    &\left\|\lambda_i^{[k+1]}-\hat{\lambda}_i^{[k+1]}\right\|_\infty \leq 2(2\rho+\sqrt{n}L_i)\Delta.\label{error:3}
\end{align}
\end{subequations}

\end{lemma}
\textit{Proof.} For \eqref{error:1} and \eqref{error:2}, see \cite[Lemma~1]{rikos2023distributed2}; for \eqref{error:3}, see Appendix~\ref{app 2}. \hfill$\blacksquare$ 

\vspace{.3cm} 

Following the analysis in \cite{jiang2021asynchronous}, we define the positive scalars $0<M_z<\infty$, such that $\left\|z^{[k]}-z^*\right\|\leq M_z$, to simplify our later analysis. The main convergence result is formally stated in the following theorem.
\begin{theorem}[Global Linear Convergence of DQLCA]\label{thm: convergence}
Suppose Assumptions \ref{ass:3} and \ref{ass:1} hold, and let $z^*$ be the optimal solution to problem \eqref{eq: quantized problem}. 
If the penalty parameter $\rho$ satisfies the condition:
\begin{equation}\label{eq: rho condition}
    \rho > \frac{\left(\sum_{i=1}^N L_i\right)^2}{N\sum_{i=1}^N \mu_i},
\end{equation}
then the sequence $\{\hat z_i^k\}$ generated by Algorithm \ref{alg: DQLCA} converges globally and linearly to the \ADDBoyu{neighborhood} of optimal solution $z^*$, such that:
\begin{equation}
    \left\| \hat z_i^{k+1} - z^*\right\|^2 \leq \theta \left\| \hat z_i^k - z^*\right\|^2+\mathcal{O}(\Delta),
\end{equation}
where $\mathcal{O}(\Delta)=4M_z\left(2+\sqrt{n}\sum_{i=1}^NL_i/\left(\rho N\right)\right)\Delta$. \ADDBoYu{Furthermore, $\theta \in (0, 1)$ represents the linear convergence rate, which is defined as $\theta = 1 - \frac{2\sum_{i=1}^N \mu_i}{\rho N} + \frac{2\left(\sum_{i=1}^N L_i\right)^2}{\rho^2 N^2}$.} 
\end{theorem}
\textit{Proof.} 
See Appendix \ref{app 3}. \hfill$\blacksquare$ 

\vspace{.3cm}

Theorem~\ref{thm: convergence} establishes that Algorithm~\ref{alg: DQLCA} achieves global linear convergence to a bounded neighborhood of the exact optimal solution. 
\ADDBoYu{Specifically, if the quantization level approaches to zero ($\Delta \to 0$), this error neighborhood completely vanishes, allowing the proposed algorithm to converge to the exact optimal solution. In this scenario, our approach essentially functions as an inexact extension of the \emph{D-ADMM-FTERC} algorithm presented in \cite{jiang2021distributed}. 
To explicitly highlight this relationship, we will utilize this inexact \emph{D-ADMM-FTERC} variant as a baseline for comparative performance analysis in the numerical experiments presented in the next section.}

\begin{remark}
It is worth highlighting the distinction between 
the convergence rate in Theorem \ref{thm: convergence} and prior results in \cite{rikos2023asynchronous}. 
The authors in \cite{rikos2023asynchronous} employ an exact C-ADMM framework paired with finite-time quantized averaging, which guarantees a sublinear convergence rate typical for general convex objective functions. 
In contrast, by explicitly leveraging the $\mu_i$-strong convexity and $L_i$-smoothness of the local cost functions, our proposed DQLCA algorithm guarantees global linear convergence. 
\ADDBoYu{Notably, this accelerated theoretical rate is achieved while simultaneously decreasing local computation costs. Exact C-ADMM requires solving a computationally expensive subproblem at each iteration, which typically relies on an external solver. By replacing this with a simple linearized primal update, our method substantially reduces the computational burden at each node.}
\end{remark}

\section{Numerical Simulation}

In this section, we present simulation results to demonstrate the performance of our algorithm. 
To highlight the computational improvements introduced by our proposed method, \ADDBoYu{we compare it against the exact unquantized and quantized C-ADMM frameworks proposed in \cite{jiang2021distributed} and \cite{rikos2023asynchronous}, respectively}.

We consider a decentralized network consisting of $N=50$ nodes, modeled as a strongly connected digraph. 
Each node $i$ has a local cost function $f_i(x_i)=0.5 x_i^\top P_i x_i + q_i^\top x_i$, where $x_i \in \mathbb{R}^n, P_i \in \mathbb{S}_{++}^n$ and $q_i \in \mathbb{R}^n$ for all $i \in \{{1,2, ..., N}\}$. 
To ensure the strong convexity and $L-$smoothness conditions required by Assumption~\ref{ass:1}, our problem parameters are generated as follows. 
For each node $i$, (i) $P_i$ was initialized as the square of a randomly generated symmetric matrix $A_i$ to ensure its positive definite property, (ii) $q_i$ was initialized as the negation of the product of the transpose of $A_i$ and a randomly generated vector $b_i$, (iii) a penalty parameter $\rho$ that satisfies the theoretical lower bound established in Theorem \ref{thm: convergence} is selected. 
Using this setup, we execute Algorithm~\ref{alg: DQLCA} to demonstrate how the decision variables converge to the global optimal solution across varying quantization levels: $\Delta=\{10^{-3}, 10^{-4}, 10^{-5}, 10^{-6}\}$. 
To quantify the performance, we track the error at each iteration $k$, defined as $ \sum_{i=1}^N \left\|\hat z_i^{[k]}-z^*\right\|_{\infty}$, where $z^*$ represents the exact optimal solution of problem \eqref{eq: quantized problem}. \ADDBoYu{The results are plotted on a logarithmic scale in Fig.~\ref{fig: result}.} 

\begin{figure}[ht]
	\centering
\includegraphics[width=0.52\textwidth,height=0.315\textheight]{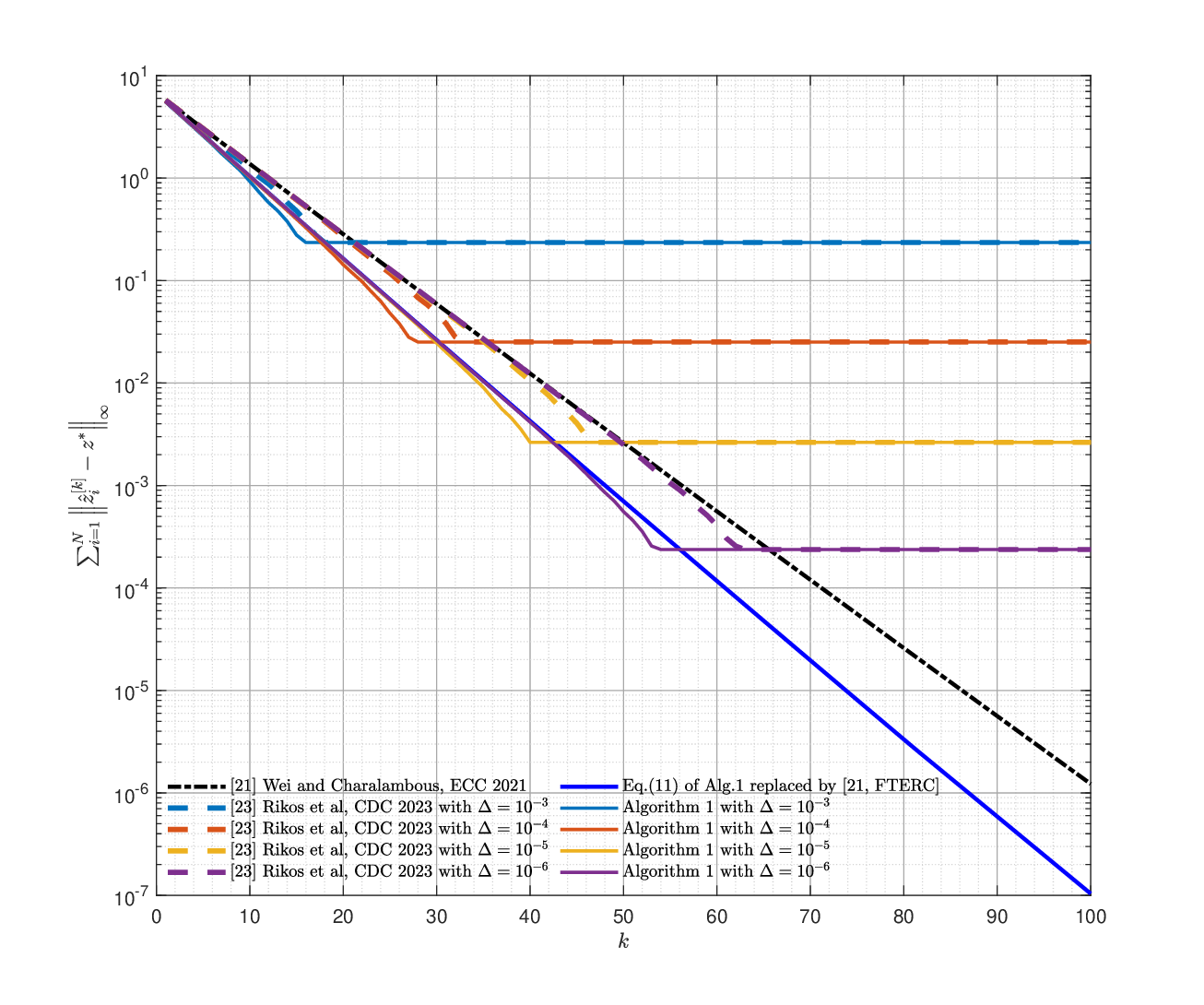}
	\caption{Performance comparison of Algorithm~\ref{alg: DQLCA}, D-ADMM-FTERC~\cite{jiang2021distributed}, and QuAsyADMM~\cite{rikos2023asynchronous} over a directed graph with communication quantization levels \(\Delta = 10^{-3}, 10^{-4}, 10^{-5}, 10^{-6}\).}
	\label{fig: result}
\end{figure}
As illustrated in Fig. \ref{fig: result}, Algorithm \ref{alg: DQLCA} successfully converges to a neighborhood of the optimal solution, the size of which is directly governed by the quantization level $\Delta$. 
Specifically, a smaller $\Delta$ yields a tighter approximation of the exact optimal solution. 
This behavior validates the theoretical bound established in Theorem~\ref{thm: convergence}. 
\ADDBoYu{To evaluate this performance, we benchmark Algorithm \ref{alg: DQLCA} against two exact C-ADMM methods: the unquantized D-ADMM-FTERC in \cite{jiang2021distributed} (black dashed line) and the quantized QuAsyADMM in \cite{rikos2023asynchronous} (colored dashed lines). The results demonstrate that our algorithm achieves the same error floors as the QuAsyADMM, but exhibits a noticeably faster convergence rate. Moreover, Algorithm \ref{alg: DQLCA} achieves accelerated performance while utilizing the linearized primal update, which significantly reduces the computational burden at each node. Finally, to explicitly demonstrate communication efficiency, we compare our method against its own unquantized baseline ($\Delta \to0$, solid blue line), which functions as an inexact extension of D-ADMM-FTERC. While this unquantized baseline reaches the exact optimum, it requires transmitting high-precision continuous variables that can easily overwhelm limited bandwidth. Instead, by accepting a small bounded error floor, our quantized algorithm substantially reduces the transmitted message size, making it practical for deployment in real-world, bandwidth-constrained environments.}



\section{Conclusions and Future Directions}

In this paper, we proposed DQLCA, a novel decentralized optimization algorithm that integrates IC-ADMM  with a finite time quantized averaging protocol. 
We showed that the proposed method enables the computation of the global optimal solution in a fully decentralized manner over directed communication networks. 
Additionally, we established that DQLCA achieves global linear convergence to a neighborhood of the exact optimum that depends on the utilized quantization level. 
Numerical simulations highlight our theoretical findings and demonstrate the algorithm’s efficiency. 

Future research directions will focus on extending the DQLCA framework to accommodate time-varying directed communication topologies, asynchronous update protocols, and operation under relaxed assumptions, such as non-convex local cost functions.



\appendices
\section{Proof of Lemma \ref{lemma:1}}\label{app 1}
By substituting the primal update \eqref{eq: new local primal} into the dual update \eqref{eq: global estimation}, we can isolate the gradient term:
\begin{equation}
    \begin{split}
   \hat \lambda_i^{[k+1]} =& \hat\lambda_i^{[k]} + \rho \left(\hat x_i^{[k+1]} - \hat z_i^{[k+1]}\right)\\
  \overset{\eqref{eq: new local primal}}{=}  & \hat\lambda_i^{[k]} + \rho \left(\hat z_i^{[k]} - \frac{1}{\rho}\left( \nabla f_i\left(\hat z_i^{[k]}\right)+\hat\lambda_i^{[k]}\right) - \hat z_i^{[k+1]}\right)\\
  =& \rho\left(\hat z_i^{[k]} - \hat z_i^{[k+1]}\right) - \nabla f_i\left(\hat z_i^{[k]}\right).
    \end{split}
\end{equation}
Rearranging this relationship yields Lemma \ref{lemma:1}, which concludes the proof.
\section{Proof of Lemma \ref{lemma:2}}\label{app 2}
Here we analyze the difference between the exact dual update and the quantized dual update. \eqref{eq:step3-closed} and \eqref{eq: dual update 1} are used to substitute ${\lambda}_{i}^{[k+1]}$ and $\hat{\lambda}_{i}^{k+1}$ respectively:
\begin{equation} \label{lambda-error}
\begin{aligned}
   &\left \|{\lambda}_{i}^{[k+1]}-\hat{\lambda}_{i}^{k+1}\right\|_{\infty}\\
   = &\left\| \lambda_{i}^{[k]}-\hat \lambda_{i}^{[k]}+\rho\left(x_{i}^{[k+1]}-z^{[k+1]}-\hat x_{i}^{[k+1]}+\hat z_i^{[k+1]}\right)\right\|_{\infty}\\
    \leq &\left\|\lambda_{i}^{[k]}-\hat \lambda_{i}^{[k]} + \rho\left(x_{i}^{[k+1]}-\hat x_{i}^{[k+1]}\right)\right\|_{\infty}\\& + \rho\left\|z^{[k+1]}-\hat z_i^{[k+1]}\right\|_{\infty}.
\end{aligned}
\end{equation}
We split the right-hand side of \eqref{lambda-error} into two components (a), (b). For (a), we substitute the explicit linearized primal updates for $x_i^{[k+1]}$ and $\hat x_i^{[k+1]}$. Notably, the previous dual variables $\lambda_i^{[k]}$ and $\hat \lambda_i^{[k]}$ perfectly cancel out during this substitution. To bound the resulting gradient difference, we apply the $L_i$-smoothness of the local cost functions alongside the standard dimensional norm equivalence $\scriptstyle\|x\|_\infty \leq \|x\|_2 \leq \sqrt{n}\|x\|_\infty$. Given \eqref{error:2}, we obtain:
\begin{equation}
\begin{aligned}
    &\left\|\lambda_{i}^{[k]}-\hat \lambda_{i}^{[k]} + \rho\left(x_{i}^{[k+1]}-\hat x_{i}^{[k+1]}\right)\right\|_{\infty}\\
    =\;&\left\|\rho\left(z^{[k]}-\hat z_i^{[k]}\right)-\left(\nabla f_i\left(z^{[k]}\right)-\nabla f_i\left(\hat z_i^{[k]}\right)\right)\right\|_{\infty}\\
    \overset{\eqref{eq: lip}}{\leq}\;&\rho\left\|z^{[k]}-\hat z_i^{[k]}\right\|_{\infty} + \sqrt{n}L_i\left\| z^{[k]}-\hat z_i^{[k]}\right\|_{\infty}\\
    \overset{\eqref{error:2}}{\leq}& 2(\rho +\sqrt{n}L_i)\Delta.
\end{aligned}
\end{equation}
For (b), it directly yields $\scriptstyle\rho\left\|z^{[k+1]}-\hat z_i^{[k+1]}\right\|_{\infty}\leq2\rho\Delta$. Combining these two components back into the initial expansion yields the constant error bound stated in  \eqref{error:3}, which concludes the proof.

\section{Proof of theorem \ref{thm: convergence}}\label{app 3}
From the dual optimality condition for the consensus distributed optimization problem \cite[chapter 7.2]{boyd2011distributed}, we have:
\begin{equation} \label{eq:consensus dual optimum property}
    \sum_{i=1}^N \lambda_i^{[k]} = 0, \;\sum_{i=1}^N \nabla f_i(z^*) = 0 .
\end{equation}

We begin with the $\mu_i$-strong convexity of the local cost functions. Because the algorithm achieves consensus on the quantized tracking variables at each iteration (i.e., $\hat z_1^{[k]} = \dots = \hat z_N^{[k]}$), we can factor it out of the summation. By summing the standard strong convexity inequalities \eqref{eq: mu} across all nodes $i$, we establish a fundamental bound on the distance to the global optimum:

\begin{equation}\label{eq: a3e1}
    \begin{split}
    & \sum_{i=1}^N \mu_i \left\|\hat z_i^{[k]} -  z^*\right \|^2\\
        \overset{\eqref{eq: mu}}{\leq}&\left(\hat z_i^{[k]} -  z^*\right)^\top\sum_{i=1}^N \left(\nabla f_i\left(\hat z_i^{[k]}\right) - \nabla f_i\left(z^*\right)\right).
    \end{split}
\end{equation}
Substituting \eqref{eq: sub1}, \eqref{eq:consensus dual optimum property} into \eqref{eq: a3e1}, then we get:
\begin{equation}\label{eq: a3e11}\small
    \begin{split}
    &\left(\hat z_i^{[k]} -  z^*\right)^\top\sum_{i=1}^N \left(\rho\left(\hat z_i^{[k]}-\hat z_i^{[k+1]}\right)-\hat \lambda_i^{[k+1]}\right)\\
        =&\rho N\left(\hat z_i^{[k]} -  z^*\right)^\top\left(\hat z_i^{[k]}-\hat z_i^{[k+1]}\right)
        -\left(\hat z_i^{[k]} -  z^*\right)^\top\sum_{i=1}^N\hat \lambda_i^{[k+1]}.
    \end{split}
\end{equation}
Applying the standard inner product identity $\scriptstyle a^\top b = \frac{1}{2}(\|a\|^2 + \|b\|^2 - \|a-b\|^2)$ to the first term yields:
\begin{equation}\label{eq: a3e2}\small
    \begin{split}
    & \rho N\left(\hat z_i^{[k]} -  z^*\right)^\top\left(\hat z_i^{[k]}-\hat z_i^{[k+1]}\right)\\
    =&\frac{\rho N}{2}\left(\left\|\hat z_i^{[k]} -  z^*\right\|^2 + \left\|\hat z_i^{[k]}-\hat z_i^{[k+1]}\right\|^2-\left\|\hat z_i^{[k+1]} -  z^*\right\|^2\right).
    \end{split}
\end{equation}
Substituting \eqref{eq: a3e2} into \eqref{eq: a3e11} and rearrange the inequality, we obtain:
\begin{equation}\label{eq: a3e4}\small
    \begin{split}
    &\frac{\rho N}{2}\left\|\hat z_i^{[k+1]} -  z^*\right\|^2\\
        \leq&\left(\frac{\rho N}{2}-\sum_{i=1}^N \mu_i\right)\left\|\hat z_i^{[k]} -  z^*\right\|^2 + \frac{\rho N}{2}\left\|\hat z_i^{[k]}-\hat z_i^{[k+1]}\right\|^2\\
        &-\left(\hat z_i^{[k]} -  z^*\right)^\top\sum_{i=1}^N\hat \lambda_i^{[k+1]}.
    \end{split}
\end{equation}
To deal with the second term on the right-hand side of \eqref{eq: a3e4}, we leverage \eqref{eq: sub1} by taking the square of its summation over the index from $1 \to N$, then we obtain:
\begin{equation}\small\label{eq: a3e5}
    \begin{split}
        & \rho^2N^2\left\|\hat z_i^{[k]}-\hat z_i^{[k+1]}\right\|^2\\
        =&\left\|\sum_{i=1}^N\nabla f_i\left(\hat z_i^{[k]}\right) + \sum_{i=1}^N\hat \lambda_i^{[k+1]}\right\|^2\\
        \overset{\eqref{eq:consensus dual optimum property}}{\leq}&2\left\|\sum_{i=1}^N\left(\nabla f_i\left(\hat z_i^{[k]}\right) -\nabla f_i\left(z^*\right)\right)\right\|^2+2\left\|\sum_{i=1}^N\hat \lambda_i^{[k+1]}\right\|^2\\
        \overset{\eqref{eq: lip}}{\leq}& 2\left(\sum_{i=1}^NL_i\right)^2\left\|\hat z_i^{[k]}-z^*\right\|^2 +2\left\|\sum_{i=1}^N\hat \lambda_i^{[k+1]}\right\|^2.
    \end{split}
\end{equation}
Substituting \eqref{eq: a3e5}, scaled by $\frac{1}{2\rho N}$, into \eqref{eq: a3e4}, we obtain:
\begin{equation}\label{eq: a3e6}\small
    \begin{split}
    &\frac{\rho N}{2}\left(\left\|\hat z_i^{[k+1]} -  z^*\right\|^2\right)\\
        \leq&\left(\frac{\rho N}{2}-\sum_{i=1}^N \mu_i + \frac{\left(\sum_{i=1}^NL_i\right)^2}{\rho N}\right)\left\|\hat z_i^{[k]} -  z^*\right\|^2 \\
        &-\left(\hat z_i^{[k]} -  z^*\right)^\top\sum_{i=1}^N\hat \lambda_i^{[k+1]}+\frac{1}{\rho N}\left\|\sum_{i=1}^N\hat \lambda_i^{[k+1]}\right\|^2.
    \end{split}
\end{equation}
According to \eqref{error:2}, \eqref{error:3} of lemma \ref{lemma:2} and \eqref{eq:consensus dual optimum property}, we deal with the second and third term of right-hand side of \eqref{eq: a3e6}:
\begin{equation}\label{eq: a3e7}\small
\begin{split}
    &\left(\hat z_i^{[k]} -  z^*\right)^\top\sum_{i=1}^N\hat \lambda_i^{[k+1]}\\
    \overset{\eqref{eq:consensus dual optimum property}}{=}\quad&\left(\hat z_i^{[k]} - z^{[k]} + \left(z^{[k]} -z^*\right)\right)^\top\sum_{i=1}^N\left(\hat \lambda_i^{[k+1]}-\lambda_i^{[k+1]}\right)\\
    \overset{\eqref{error:2},\eqref{error:3}}{\leq}&2(2\Delta + M_z)\left(2\rho N+\sqrt{n}\sum_{i=1}^NL_i\right)\Delta,
\end{split}
\end{equation}
and
\begin{equation}\label{eq: a3e8}
    \begin{split}
        \left\|\sum_{i=1}^N\hat \lambda_i^{[k+1]}\right\|^2\overset{\eqref{eq:consensus dual optimum property}}{=}&\left\|\sum_{i=1}^N\left(\hat \lambda_i^{[k+1]}-\lambda_i^{[k+1]}\right)\right\|^2\\
        \overset{\eqref{error:3}}{\leq}&4n\left(2N\rho+\sqrt{n}\sum_{i=1}^NL_i\right)^2\Delta^2.
    \end{split}
\end{equation}
Finally, substituting \eqref{eq: a3e7} and \eqref{eq: a3e8} into \eqref{eq: a3e6}, scaled by $\frac{2}{\rho N}$, then we arrive at the final convergence bound:
\begin{equation}\label{eq: a3e9}
    \begin{split}
        &\left\|\hat z_i^{[k+1]} -  z^*\right\|^2\\
        \leq&\hspace{-.1cm}\underbrace{\left(1-\frac{2\sum_{i=1}^N \mu_i}{\rho N}+\frac{2\left(\sum_{i=1}^NL_i\right)^2}{\rho^2N^2}\right)}_{\theta}\hspace{-.1cm}\left\|\hat z_i^{[k]} -  z^*\right\|^2 \hspace{-.2cm}+\mathcal{O}(\Delta),
    \end{split}
\end{equation}
where $\mathcal{O}(\Delta)=4M_z\left(2+\sqrt{n}\sum_{i=1}^NL_i/\left(\rho N\right)\right)\Delta$. Note that in \eqref{eq: a3e9} we omitted the higher-order terms of $\Delta$, as they are significantly smaller than the first-order terms and can be considered negligible. 
Provided that  $\rho > \frac{\left(\sum_{i=1}^N L_i\right)^2}{N\sum_{i=1}^N \mu_i}$, the constriction factor
 $\theta$ in \eqref{eq: a3e9} is guaranteed to satisfy  $0<\theta<1$. This condition ensures the linear convergence of the algorithm. Consequently, Theorem \ref{thm: convergence} is proved.\\

\bibliographystyle{IEEEtran}   
\bibliography{references}      

@book{boyd2011distributed,
	title={Distributed optimization and statistical learning via the alternating direction method of multipliers},
	author={Boyd, Stephen and Parikh, Neal and Chu, Eric},
	year={2011},
	publisher={Now Publishers Inc}
}

@article{rikos2023distributed2,
  title={Distributed optimization with gradient descent and quantized communication},
  author={Rikos, Apostolos I and Jiang, Wei and Charalambous, Themistoklis and Johansson, Karl H},
  journal={IFAC-PapersOnLine},
  volume={56},
  number={2},
  pages={5900--5906},
  year={2023},
  publisher={Elsevier}
}

@inproceedings{du2025cdc,
title={Decentralized optimization via {RC-ALADIN} with efficient quantized communication},
author={Du, Xu and Johansson, Karl H and Rikos, Apostolos I},
booktitle={2025 IEEE 64th Conference on Decision and Control (CDC)},
pages={4357–4363},
year={2025},
organization={IEEE}
}

@article{dist_structure3,
    title={A Decentralized Framework for Multi-Agent Robotic Systems},
    author={A. Jiménez and Vicente García Díaz and Sandro Javier Bolaños Castro},
    journal={Sensors (Basel, Switzerland)},
    year={2018},
    volume={18},
    doi={10.3390/s18020417}}

@article{2024_doostmohammadian_rikos_Johansson_survey,
title = {Survey of distributed algorithms for resource allocation over multi-agent systems},
journal = {Annual Reviews in Control},
author={Doostmohammadian, Mohammadreza and Aghasi, Alireza and Pirani, Mohammad and Nekouei, Ehsan and Zarrabi, Houman and Keypour, Reza and Rikos, Apostolos I. and Johansson, Karl H.},
volume = {59}, 
pages = {100983}, 
year={2025}
}

@article{lanza2025distributed,
  title={Distributed optimization for energy grids: a tutorial on {ADMM} and {ALADIN}},
  author={Lanza, Lukas and Faulwasser, Timm and Worthmann, Karl},
  journal={System level control and optimisation of microgrids},
  pages={121--145},
  year={2025}
}

@article{dist_structure4,
    title={Blockchain meets machine learning: a survey},
    author={S. Kayikci and T. Khoshgoftaar},
    journal={Journal of Big Data},
    year={2024},
    volume={11},
    pages={1-29},
    doi={10.1186/s40537-023-00852-y}}

@article{dist_structure5,
  title={Distributed learning applications in power systems: A review of methods, gaps, and challenges},
  author={Gholizadeh, Nastaran and Musilek, Petr},
  journal={Energies},
  volume={14},
  number={12},
  pages={3654},
  year={2021},
  publisher={MDPI}
}

@article{hong2016convergence,
  title={Convergence analysis of alternating direction method of multipliers for a family of nonconvex problems},
  author={Hong, Mingyi and Luo, Zhi-Quan and Razaviyayn, Meisam},
  journal={SIAM Journal on Optimization},
  volume={26},
  number={1},
  pages={337--364},
  year={2016},
  publisher={SIAM}
}

@incollection{camacho2026introduction,
  title={Introduction to model predictive control},
  author={Camacho, Eduardo F and Bordons, Carlos and Maestre, Jos{\'e} M},
  booktitle={Model Predictive Control},
  pages={1--23},
  year={2026},
  publisher={Springer}
}

@article{zhang2025enabling,
  title={Enabling scalable distributed beamforming via networked LEO satellites towards 6G},
  author={Zhang, Yuchen and Al-Naffouri, Tareq Y},
  journal={IEEE Transactions on Wireless Communications},
  year={2025},
  publisher={IEEE}
}

@article{zhou2026preconditioned,
  title={Preconditioned inexact stochastic {ADMM} for deep models},
  author={Zhou, Shenglong and Wang, Ouya and Luo, Ziyan and Zhu, Yongxu and Li, Geoffrey Ye},
  journal={Nature Machine Intelligence},
  pages={1--12},
  year={2026},
  publisher={Nature Publishing Group UK London}
}

@article{boyd2007notes,
  title={Notes on decomposition methods},
  author={Boyd, Stephen and Xiao, Lin and Mutapcic, Almir and Mattingley, Jacob},
  journal={Notes for EE364B, Stanford University},
  volume={635},
  pages={1--36},
  year={2007}
}

@ARTICLE{11373371,
  author={Feng, Zhangcheng and Xu, Wenying and Cao, Jinde and Yang, Shaofu},
  journal={IEEE Transactions on Control of Network Systems}, 
  title={{IDO-ADMM}: Inexact Distributed Online Alternating Direction Method of Multipliers Under Full-Information and Bandit Feedback}, 
  year={2026},
  volume={},
  number={},
  pages={1-12},
  doi={10.1109/TCNS.2026.3662099}}

@article{xu2023qc,
  title={{QC-ODKLA}: Quantized and communication-censored online decentralized kernel learning via linearized {ADMM}},
  author={Xu, Ping and Wang, Yue and Chen, Xiang and Tian, Zhi},
  journal={IEEE transactions on neural networks and learning systems},
  volume={35},
  number={12},
  pages={17987--17999},
  year={2023},
  publisher={IEEE}
}

@article{li2019communication,
  title={Communication-censored linearized {ADMM} for decentralized consensus optimization},
  author={Li, Weiyu and Liu, Yaohua and Tian, Zhi and Ling, Qing},
  journal={IEEE Transactions on Signal and Information Processing over Networks},
  volume={6},
  pages={18--34},
  year={2019},
  publisher={IEEE}
}

@article{ling2015dlm,
  title={{DLM}: Decentralized linearized alternating direction method of multipliers},
  author={Ling, Qing and Shi, Wei and Wu, Gang and Ribeiro, Alejandro},
  journal={IEEE Transactions on Signal Processing},
  volume={63},
  number={15},
  pages={4051--4064},
  year={2015},
  publisher={IEEE}
}

@article{mokhtari2016dqm,
  title={{DQM}: Decentralized quadratically approximated alternating direction method of multipliers},
  author={Mokhtari, Aryan and Shi, Wei and Ling, Qing and Ribeiro, Alejandro},
  journal={IEEE Transactions on Signal Processing},
  volume={64},
  number={19},
  pages={5158--5173},
  year={2016},
  publisher={IEEE}
}

@article{yue2025differentially,
  title={Differentially private linearized {ADMM} algorithm for decentralized nonconvex optimization},
  author={Yue, Xiao-Yu and Xiao, Jiang-Wen and Liu, Xiao-Kang and Wang, Yan-Wu},
  journal={IEEE Transactions on Information Forensics and Security},
  year={2025},
  publisher={IEEE}
}

@article{cao2018distributed,
  title={Distributed linearized {ADMM} for network cost minimization},
  author={Cao, Xuanyu and Liu, KJ Ray},
  journal={IEEE Transactions on Signal and Information Processing over Networks},
  volume={4},
  number={3},
  pages={626--638},
  year={2018},
  publisher={IEEE}
}

@article{ng2011inexact,
  title={Inexact alternating direction methods for image recovery},
  author={Ng, Michael K and Wang, Fan and Yuan, Xiaoming},
  journal={SIAM Journal on Scientific Computing},
  volume={33},
  number={4},
  pages={1643--1668},
  year={2011},
  publisher={SIAM}
}

@article{chen2017efficient,
  title={An efficient inexact symmetric {Gauss--Seidel} based majorized {ADMM} for high-dimensional convex composite conic programming},
  author={Chen, Liang and Sun, Defeng and Toh, Kim-Chuan},
  journal={Mathematical Programming},
  volume={161},
  number={1},
  pages={237--270},
  year={2017},
  publisher={Springer}
}

@article{he2020optimally,
  title={Optimally linearizing the alternating direction method of multipliers for convex programming},
  author={He, Bingsheng and Ma, Feng and Yuan, Xiaoming},
  journal={Computational Optimization and Applications},
  volume={75},
  number={2},
  pages={361--388},
  year={2020},
  publisher={Springer}
}

@article{zhou2023federated,
  title={Federated learning via inexact {ADMM}},
  author={Zhou, Shenglong and Li, Geoffrey Ye},
  journal={IEEE Transactions on Pattern Analysis and Machine Intelligence},
  volume={45},
  number={8},
  pages={9699--9708},
  year={2023},
  publisher={IEEE}
}

@article{rikos2022non,
  title={Non-oscillating quantized average consensus over dynamic directed topologies},
  author={Rikos, Apostolos I and Hadjicostis, Christoforos N and Johansson, Karl H},
  journal={Automatica},
  volume={146},
  pages={110621},
  year={2022},
  publisher={Elsevier}
}

@inproceedings{rikos2023asynchronous,
  title={Asynchronous distributed optimization via {ADMM} with efficient communication},
  author={Rikos, Apostolos I and Jiang, Wei and Charalambous, Themistoklis and Johansson, Karl H},
  booktitle={2023 62nd IEEE Conference on Decision and Control (CDC)},
  pages={7002--7008},
  year={2023},
  organization={IEEE}
}

@inproceedings{jiang2021distributed,
  title={Distributed alternating direction method of multipliers using finite-time exact ratio consensus in digraphs},
  author={Jiang, Wei and Charalambous, Themistoklis},
  booktitle={2021 European Control Conference (ECC)},
  pages={2205--2212},
  year={2021},
  organization={IEEE}
}

@misc{jang2026aliaadaptivelinearizedadmm,
      title={{ALiA}: Adaptive Linearized {ADMM}}, 
      author={Uijeong Jang and Kaizhao Sun and Wotao Yin and Ernest K Ryu},
      year={2026},
      eprint={2602.15000},
      archivePrefix={arXiv},
      primaryClass={math.OC},
      url={https://arxiv.org/abs/2602.15000}, 
}

@inproceedings{jiang2021asynchronous,
  title={An asynchronous approximate distributed alternating direction method of multipliers in digraphs},
  author={Jiang, Wei and Grammenos, Andreas and Kalyvianaki, Evangelia and Charalambous, Themistoklis},
  booktitle={2021 60th IEEE Conference on Decision and Control (CDC)},
  pages={3406--3413},
  year={2021},
  organization={IEEE}
}

@article{rikos2025distributed,
  title={Distributed optimization with efficient communication, event-triggered solution enhancement, and operation stopping},
  author={Rikos, Apostolos I and Jiang, Wei and Charalambous, Themistoklis and Johansson, Karl H},
  journal={arXiv preprint arXiv:2504.16477},
  year={2025}
}

@article{oliva2016distributed,
  title={Distributed finite-time calculation of node eccentricities, graph radius and graph diameter},
  author={Oliva, Gabriele and Setola, Roberto and Hadjicostis, Christoforos N},
  journal={Systems \& Control Letters},
  volume={92},
  pages={20--27},
  year={2016},
  publisher={Elsevier}
}

\end{document}